 \documentclass[a4paper,11pt]{article}
 \usepackage[plainpages=false]{hyperref}
 \usepackage{amsfonts,amsmath,amssymb,amsthm}
 \usepackage{latexsym,lscape,rawfonts,mathrsfs}

 \usepackage[dvips]{color}
 \usepackage{multicol}

 \usepackage{natbib}

 \usepackage[all]{xy}
 \usepackage{eufrak}
 \usepackage{makeidx}
 \usepackage{graphicx,psfrag}

 \usepackage{array,tabularx}

 \usepackage{setspace}

 \newcommand{\C}{\ensuremath{\mathbb{C}}}
 \newcommand{\D}[2]{\ensuremath{ \frac{\partial{#1}}{\partial{#2}}}}

 \newcommand{\R}{\ensuremath{\mathbb{R}}}
 \newcommand{\Z}{\ensuremath{\mathbb{Z}}}
 \newcommand{\CP}{\ensuremath{\mathbb{CP}}}
 \newcommand{\st}{\ensuremath{\sqrt{-1}}}

 \newcommand{\KRf}{K\"ahler Ricci flow\;}
 \newcommand{\KRfd}{K\"ahler Ricci flow.\;}
 \newcommand{\KRF}{K\"ahler Ricci Flow\;}

 \DeclareMathOperator{\Vol}{Vol}
 \DeclareMathOperator{\diam}{diam}

 \newcommand{\Blow}[1]{\ensuremath{\mathbb{CP}^2 \# {#1}\overline{\mathbb{CP}}^2}}
 \newcommand{\PxP}{\ensuremath{\mathbb{CP}^1 \times \mathbb{CP}^1}}
 
 \newcommand{\norm}[2]{{ \ensuremath{\|} #1 \ensuremath{\|}}_{#2}}
 
 \newcommand{\sconv}{\ensuremath{ \stackrel{C^\infty}{\longrightarrow}}}

 \makeatletter
 \def\ExtendSymbol#1#2#3#4#5{\ext@arrow 0099{\arrowfill@#1#2#3}{#4}{#5}}
 
 \makeatother

 \definecolor{hao}{rgb}{1,0.5,0}
 \definecolor{miao}{cmyk}{0.5,0,0.2,0.2}
 \definecolor{qiao}{gray}{0.96}

 \newtheorem{proposition}{Proposition}[section]
 \newtheorem{lemma}{Lemma}[section]

 \newtheorem{definition}{Definition}[section]

 \newtheorem{remark}{Remark}[section]

 \newtheorem{theoremin}{Theorem}
 \newtheorem{corollaryin}{Corollary}

 \title{K\"{a}hler Ricci Flow on Fano Surfaces (I)}
 \author{Xiuxiong Chen\footnote{Partially supported by a NSF grant.}\;,  Bing Wang}
 \date{}
 
\begin{document}
 \maketitle

 \begin{abstract}
   We show the properties of the blowup limits of \KRf solutions on
   Fano  surfaces if Riemannian curvature is unbounded.
    As an application, on every toric Fano surface, we prove that
     \KRf converges to a K\"ahler Ricci soliton metric
   if the initial metric has toric symmetry. Therefore we give a new
   Ricci flow proof of existence of K\"ahler Ricci soliton metrics
   on toric surfaces.
\end{abstract}

  \section{Introduction}

 This is the first part of our study of K\"ahler Ricci flow on Fano
 surfaces.  In this note, we study the convergence of
 \KRf on Fano surfaces if Riemannian curvature is uniformly bounded and
 we discuss the methods to obtain the Riemannian curvature bound.\\

 In~\cite{Ha1}, Hamilton defined Ricci flow and applied it to prove
 that every simply connected 3-manifold with positive Ricci curvature metric admits
 a constant curvature metric, hence it is  diffeomorphic to $S^3$.  From then on, Ricci flow
 became a
 powerful tool to search Einstein metrics on manifolds. If the
 underlying manifold is a K\"ahler manifold whose first Chern class has definite sign,
 then the normalized Ricci flow
 is called the K\"ahler Ricci flow.  It was proved by Cao~\cite{Cao85}, who
 followed Yau's fundamental estimates, that \KRf always exists
 globally.  If the first Chern class of the underlying manifold is
 negative or null, Cao showed that \KRf will converge to  a K\"ahler Einstein (KE) metric.
 If the first Chern class of the underlying
 manifold is positive,  then the convergence is much harder and still not clearly now.
 The first breakthrough in this
 direction is the work of~\cite{CT1} and~\cite{CT2}. There Chen and
 Tian showed that the \KRf converges to a KE metric if the initial
 metric has positive bisectional curvature.  Around 2002, Perelman
 made some fundamental estimates along K\"ahler Ricci flow,
 he showed that scalar
 curvatures, diameters and normalized Ricci potentials are all
 uniformly bounded along every \KRfd Together with his
 no-local-collapsing theorem, these estimates give us a lot of
 information.   After Perelman's fundamental estimates,
 there are numerous works concerning the convergence of K\"ahler
 Ricci flow.   Due to the limited knowledge of authors, we'll not
 give a complete list of all the contributors.\\

  In this note, we only consider the convergence of \KRf on Fano
  surfaces. If we assume Riemannian curvature is uniformly bounded along the flow, then
  we can show that this flow converges to a K\"ahler Ricci soliton (KRS) metric in the same
  complex structure.
 \begin{theoremin}
   Suppose $\{(M, g(t)), 0 \leq t < \infty \}$ is a \KRf solution on
   a Fano surface $M$, $J$ is the complex structure compatible with $(M, g(t))$.
   If Riemannian curvature is uniformly
   bounded along this flow,  then for every sequence $t_i \to \infty$,
   there is a subsequence of times $t_{i_k}$ and diffeomorphisms
   $\Phi_{i_k}: M \to M$ such that
  \begin{align*}
  \Phi_{i_k}^*g(t_{i_k}) \sconv  h,  \quad
  (\Phi_{i_k})_{*}^{-1} \circ J \circ (\Phi_{i_k})_{*}  \sconv J.
  \end{align*}
  under a fixed gauge.  Here $h$ is a K\"ahler metric compatible with complex structure
  $J$ and $(M, h)$ is a K\"ahler Ricci soliton metric, i.e., there
  exists a smooth function $f$ on $M$ such that
  \begin{align*}
         Ric_h - h= \mathscr{L}_{\nabla f} h.
  \end{align*}
 \label{theoremin: bounded}
 \end{theoremin}

 In general dimension K\"ahler Ricci flow,
 even if Riemannian curvature is uniformly bounded,
 we can only obtain $(\Phi_{i_k})_{*}^{-1} \circ J \circ (\Phi_{i_k})_{*}  \sconv \tilde{J}$
 for some complex structure $\tilde{J}$ compatible with $h$.
 $\tilde{J}$ may be different from $J$.
 The reason that we can obtain the same $J$ here is that we are
 dealing with Fano surfaces now.   The classification of Fano surfaces
 gives us a lot of information about the limit complex structures.\\

  In order to set up a uniform Riemannian curvature bound, we
  observe that there are actually some topological and geometric
   obstructions for
  Riemannian curvature to be unbounded  along the flow.  Because if the
  Riemannian curvature is not uniformly bounded, we can blowup the
  flow at maximal Riemannian curvature points. We denote such blowup
  limits as ``deepest bubbles" and find that they satisfy strong
  topological and geometric conditions.

  \begin{theoremin}
    Suppose $\{(M, g(t)), 0 \leq t < \infty\}$ is a \KRf solution on a
  Fano surface $M$. If Riemannian curvature is not uniformly
  bounded, then every deepest bubble $X_{\infty}$ is a finite
  quotient of a hyper K\"ahler ALE manifold.   Moreover, $X_{\infty}$
  doesn't contain any compact divisor.
  \label{theoremin: bubble}
  \end{theoremin}
  Part of this theorem was obtained independently in~\cite{RZZ}.

  On one hand, $X_{\infty}$ has particular topology and geometry.
  On the other hand, all the topology and geometry of $X_{\infty}$
  come out from the underlying manifolds.
  Therefore, on a Fano surface with simple topology and geometry,
  it's plausible that the Riemannian curvature is uniformly bounded.
  Actually, we can prove the following theorem.

 \begin{theoremin}
  Suppose $\{(M,g(t)), 0 \leq t < \infty\}$ is a \KRf solution on a
  toric Fano surface.  If the initial metric is toric symmetric,
  then the Riemannian curvature is uniformly bounded along this
  flow.
 \label{theoremin: toric}
 \end{theoremin}

  This result is not new.  In \cite{Zhu}, Zhu showed that staring
  from any toric symmetric metric, \KRf will converge to a KRS metric
  on every toric manifold.
  He used complex Monge-Ampere equation theory and reduce the
  convergence of \KRf to the control of $C^0$-norm of
  Ricci potential functions on $M$.  Using the toric condition, he
  is able to obtain the required $C^0$-estimates.
  Since Theorem~\ref{theoremin: toric} only concerns toric Fano surfaces, it's not surprising that our
  proof is simpler and more geometric.   In fact, toric Fano
  surfaces with toric symmetric metrics admit almost the simplest
  topology and geometry one can imagine for Fano surfaces.  Because
  of this simplicity,  the formation of deepest bubbles is prevent,
  therefore the Riemannian curvature must be bounded along the flow.  This ruling-out-bubble idea originates from \cite{CLW}.
 However, in \cite{CLW}, explicit energy bounds are calculated to exclude
 deepest bubbles.  This calculation is avoided in our
 cases and we can exclude deepest bubbles directly by topological
 and geometric obstructions.\\

 As the combination of Theorem~\ref{theoremin: bounded} and
 Theorem~\ref{theoremin: toric}, we know every toric Fano surface
 admits a KRS metric.  Recall that a KRS metric becomes a KE
 metric if and only if the Futaki invariant of the underlying
 manifold vanishes. By the classification of Fano surfaces, every
 toric Fano surface must be one of the following types: $\PxP, \CP^2, \Blow{k}(1 \leq k \leq 3)$.
 Only $\Blow{}$ and $\Blow{2}$ have nonvanishing Futaki invariants.
 Therefore, we have
 \begin{corollaryin}
   There exist nontrivial KRS metrics on $\Blow{}$ and $\Blow{2}$.
   There exist KE metrics on $\PxP$, $\CP^2$ and $\Blow{3}$.
 \end{corollaryin}

 This result is well known although our proof is new.  The existence of KRS on
  $\Blow{}$ was first proved by Koiso (\cite{Ko}).  The existence of KE metric
  on $\Blow{3}$ was first proved by Tian and Yau(\cite{TY}).     For
  a general toric manifold, the existence of KRS metric was proved by
  X. Wang and Zhu (\cite{WZ}).\\

 The organization of this note is as follows: In section 2, we setup
 the basic notations.  Then we prove
 Theorem~\ref{theoremin: bounded}, Theorem~\ref{theoremin: bubble} and
 Theorem~\ref{theoremin: toric} in section 3,  section 4 and  section 5
 respectively.

 {\bf Acknowledgments:}   The first named author wants to thank S. K.
 Donaldson, G. Tian for discussions in related topics. The second
 author wants to thank G. Tian, J. Cheeger, J. Lott, K. Grove, B.
 Chow for their
 interest in this work.\\

 \section{Setup of notations}

 Let $(M,g,J)$ be an $n$-dimensional compact complex manifold with
 Riemannian metric $g$, $\omega$ be the form compatible with $g$
 and $J$.  $(M,g,J)$ is K\"ahler if and only if $\nabla J \equiv 0$.
 This holds if and only if $\omega$ is a positive closed
 $(1,1)$-form. We'll only study K\"ahler manifold in this note.
  In local complex coordinates $\{z_1, \cdots, z_n\}$, the metric form
 $\omega$ is of the form
 \[
 \omega = \sqrt{-1} \displaystyle \sum_{i,j=1}^n\;g_{i \overline{j}}
 d\,z^i\wedge d\,z^{\overline{j}}  > 0,
 \]
 where $\{g_{i\overline {j}}\}$ is a positive definite Hermitian
 matrix function. The K\"ahler condition requires that $\omega$ is a
 closed positive (1,1)-form.  Given a K\"ahler metric $\omega$, the curvature tensor
 is
 \[
  R_{i \overline{j} k \overline{l}} = - {{\partial^2 g_{i \overline{j}}} \over
 {\partial z^{k} \partial z^{\overline{l}}}} + \displaystyle
 \sum_{p,q=1}^n g^{p\overline{q}} {{\partial g_{i \overline{q}}}
 \over {\partial z^{k}}}  {{\partial g_{p \overline{j}}} \over
 {\partial z^{\overline{l}}}}, \qquad\forall\;i,j,k,l=1,2,\cdots n.
 \]
 The Ricci curvature form is
 \[
  {\rm Ric}(\omega) = \sqrt{-1} \displaystyle \sum_{i,j=1}^n \;R_{i \overline{j}}(\omega)
 d\,z^i\wedge d\,z^{\overline{j}} = -\sqrt{-1} \partial
 \overline{\partial} \log \;\det (g_{k \overline{l}}).
 \]
 It is a real, closed (1,1)-form. Recall that $[\omega]$ is called a
 canonical K\"ahler class if this Ricci form is cohomologous to
 $\omega,\; $, i.e., $[Ric]=[\omega]$.\\

    Now we assume that the first Chern class $c_1(M)$ is positive.
 The normalized Ricci flow equation(c.f. \cite{Cao85}) in the canonical class
 of $M$ is
 \begin{equation}
  {{\partial g_{i \overline{j}}} \over {\partial t }} = g_{i \overline{j}}
  - R_{i \overline{j}}, \qquad\forall\; i,\; j= 1,2,\cdots ,n.
 \label{eqn: kahlerricciflow}
 \end{equation}
 It follows that on the level of K\"ahler potentials,
 the flow becomes
 \begin{equation}
   {{\partial \varphi} \over {\partial t }} =  \log {{\omega_{\varphi}}^n \over {\omega}^n } + \varphi - h_{\omega} ,
 \label{eqn: flowpotential}
 \end{equation}
 where $h_{\omega}$ is defined by
 \[
   {\rm Ric}(\omega)- \omega = \st \partial \overline{\partial} h_{\omega}, \; {\rm and}\;\displaystyle \int_X\;
   (e^{h_{\omega}} - 1)  {\omega}^n = 0.
 \]
  Along \KRF, the evolution equations of curvatures are listed in
  Table~\ref{table: cevformula}.

 \begin{table}[h]
 \begin{center}
 \begin{tabular}[b]{|c @{} l|}
 \hline
 $\displaystyle \D{}{t}R_{i \overline{j} k
 \overline{l}}$ & =$\displaystyle \triangle R_{i \overline{j} k \overline{l}} +
 R_{i \overline{j} p \overline{q}} R_{q \overline{p} k \overline{l}}
 - R_{i \overline{p} k \overline{q}} R_{p \overline{j} q
 \overline{l}} + R_{i
 \overline{l} p \overline{q}} R_{q \overline{p} k \overline{j}} + R_{i \overline{j} k
 \overline{l}}$\\
  &\quad $-{1\over 2} \left( R_{i \overline{p}}R_{p \overline{j} k
 \overline{l}}  + R_{p \overline{j}}R_{i \overline{p} k \overline{l}}
 + R_{k \overline{p}}R_{i \overline{j} p \overline{l}} + R_{p
 \overline{l}}R_{i \overline{j} k \overline{p}} \right)$.\\
 $\displaystyle \D{}{t} R_{i\bar{j}}$  &  =$\displaystyle \triangle
  R_{i\bar j} + R_{i\bar j p \bar q} R_{q \bar p} -R_{i\bar p} R_{p \bar j}$.\\
 $\displaystyle \D{}{t} R$ & =$\displaystyle \triangle R + R_{i\bar j} R_{j\bar i}- R$.\\
 \hline
 \end{tabular}
 \end{center}
 \caption{Curvature evolution equations along \KRF}
 \label{table: cevformula}
 \end{table}
 One shall note that the Laplacian operator appears in the above
 formulae is the Laplacian-Beltrami operator on functions. As usual,
 the flow equation (\ref{eqn: kahlerricciflow})
  or (\ref{eqn: flowpotential}) is referred as the K\"ahler Ricci flow on $M$. It is
 proved by Cao \cite{Cao85}, who followed Yau's celebrated work
 \cite{Yau78}, that the K\"ahler Ricci flow exists globally for
 any smooth initial K\"ahler metric.\\

 In his unpublished work, Perelman obtained some deep estimates along
 K\"{a}hler Ricci flow on Fano manifolds. The detailed proof can be
 found in Sesum and Tian's note~\citet{ST}.
  \begin{proposition}[Perelman, c.f.~\cite{ST}]
 Suppose $\{(M^n, g(t)), 0 \leq t < \infty \}$ is a \KRf solution. There
 are two positive constants $D, \kappa$ depending only on this flow such that the
 following two estimates hold.
 \begin{enumerate}
 \item  Let $R_{g_t}$ be the scalar curvature under metric $g_t$,
 $h_{\omega_{\varphi(t)}}$ be the Ricci potential of form $\omega_{\varphi(t)}$ satisfying $\frac{1}{V} \int_M e^{h_{\omega_{\varphi(t)}}} \omega_{\varphi(t)}^n=1$. Then we have
 \begin{align*}
      \norm{R_{g_t}}{C^0} + \diam_{g_t} M +
      \norm{h_{\omega_{\varphi(t)}}}{C^0} + \norm{\nabla h_{\omega_{\varphi(t)}}}{C^0} < D.
 \end{align*}
 \item   $ \displaystyle
      \frac{\Vol(B_{g_t}(x, r))}{r^{2n}} > \kappa$ for every
       $r \in (0, 1)$, $(x, t) \in M \times [0, \infty)$.
 \end{enumerate}
 \label{proposition: perelman}
 \end{proposition}
 These fundamental estimates will be used essentially in our
 arguments.

 \section{Convergence of \KRf on Fano Surfaces if Riemannian Curvature is Uniformly Bounded}
 In this section, we will prove Theorem~\ref{theoremin: bounded}.

 \begin{proof}
     Since the Riemannian curvature is uniformly bounded, diameter
     is uniformly bounded. By the
     compactness theorem of Ricci flow, we know that
     for every sequence
     of times $t_i \to \infty$, there is a subsequence $t_{i_k} \to \infty$
     such that
  \begin{align*}
     \{(M, g(t+t_{i_k})),  -t_{i_k} < t \leq 0\} \longrightarrow \{(\hat{M}, \hat{g}(t)), -\infty < t \leq
     0\}
  \end{align*}
  in Cheeger-Gromov sense.  In particular, $(M, g(t_{i_k}))$
  converges to $(\hat{M}, \hat{g}(0))$ in Cheeger-Gromov sense,
  i.e., there are diffeomorphisms $\varphi_{i_k}: \hat{M} \to M$
  such that
  \begin{align*}
      \varphi_{i_k}^*(g(t_{i_k})) \sconv \hat{g}(0).
  \end{align*}
   For the simplicity of notations,  we will rewrite this as
  $\varphi_{i}^*(g(t_{i})) \sconv \hat{g}$.
  Note that on $\hat{M}$, complex structure
  $J_{i}=(\varphi_i)_{*}^{-1} \circ J \circ (\varphi_i)_{*}$
  is compatible with $\varphi_i^* g(t_i)$ and it
  satisfies $\nabla_{\varphi_i^* g(t_i)} J_i  \equiv 0$. Taking limit, we
  have $J_i \to \hat{J}$ and $\nabla_{\hat{g}}\hat{J} \equiv 0$.
  This means that $(\hat{M}, \hat{g}, \hat{J})$ is a K\"ahler manifold.
  By the monotonicity of Perelman's $\mu$-functional, we can argue that
   $(\hat{M}, \hat{g})$ is a KRS as Natasa Sesum did in~\cite{Se}.
  On a KRS,
 $Ric_{\hat{g}} - \hat{g}= \mathscr{L}_{\nabla f} \hat{g}$ for some
 smooth function $f$.
  Clearly,
  $[\hat{R}_{i \bar{j}}] = [\hat{g}_{i \bar{j}}] >0$ and it
  implies that
  \begin{align*}
         2\pi c_1(\hat{M})^2= \Vol_{\hat{g}}(\hat{M})= \lim_{i \to
         \infty} \Vol_{g(t_i)} (M)= 2\pi c_1^2(M).
    \end{align*}
 It follows that $(\hat{M}, \hat{g})$ is a Fano surface satisfying
 $c_1^2(\hat{M})=c_1^2(M)$ and  $\hat{M}$ is diffeomorphic to
 $M$.

   By classification of Fano surfaces, every Fano surface $M$ must
   satisfy $1 \leq c_1^2(M) \leq 9$. Now we discuss cases by the values of
   $c_1^2(M)$.

 \begin{enumerate}
 \item  $9 \geq c_1^2(M)=c_1^2(\hat{M}) \geq 5$.      In this case, $M$ is diffeomorphic
 to one of $\PxP$, $\CP^2$ or $\Blow{k}(1 \leq k \leq 4)$.   By
 classification of Fano surfaces, under each of such diffeomorphic
 structure, there is only one complex structure such that $M$
 becomes Fano.     Let $\psi$ be a diffeomorphism map $\psi: M \to \hat{M}$.
 Then both $(M, J)$ and $(M, (\psi)_{*}^{-1} \circ \hat{J} \circ \psi_*)$
 are Fano manifolds with $9 \geq c_1^2(M) \geq 5$.   Therefore, we
 have $J= (\psi)_{*}^{-1} \circ \hat{J} \circ \psi_*$. Recall that
 we have
 \begin{align*}
 \varphi_i^*(g(t_i)) \sconv \hat{g}, \quad
 (\varphi_i)_{*}^{-1} \circ J \circ (\varphi_i)_{*} \sconv \hat{J}.
 \end{align*}
 It follows that
 \begin{align*}
     (\varphi_i \circ \psi)^* (g(t_i))  \sconv h= \psi^* \hat{g},
     \quad
      (\varphi_i \circ \psi)_*^{-1} \circ J \circ (\varphi_i \circ
      \psi) \sconv J=(\psi)_*^{-1} \circ \hat{J} \circ (\psi)_{*}.
 \end{align*}
 Let $\Psi_i= \varphi_i \circ \psi$, $h=\psi^* \hat{g}$,
  we finish the proof of
 Theorem~\ref{theoremin: bounded} whenever $9 \geq c_1^2(M) \geq
 5$.

 \item $4 \geq c_1^2(M)=c_1^2(\hat{M}) \geq 1$. In this case, $M$ is diffeomorphic
 to one of $\Blow{k}(5 \leq k \leq 8)$.     Under each of such
 diffeomorphic structure, there are a lot of complex structures $\tilde{J}$
 such that every $(M, \tilde{J})$ is a Fano surface.  Therefore it is
 possible that $J \neq (\psi)_*^{-1} \circ \hat{J} \circ (\psi)_*$
 and the previous argument fails.
 However, in this case, $\hat{M}$ doesn't admit any
 holomorphic vector fields. This forces the KRS metric on $\hat{M}$ to be a
 KE metric. Moreover, the first eigenvalue of $(\hat{M}, \hat{g})$ is strictly greater
 than $1$ (c.f. Lemma6.3 of~\cite{CT1}).   Note that first eigenvalues converge as the Riemannian
 manifolds converge in Cheeeger-Gromov sense,  a contradiction argument shows that the
 first eigenvalue of
 $(M, g(t))$ is strictly greater than $1$, i.e., there is a positive
 constant $\delta$ such that
 \begin{align*}
    \lambda_1(-\triangle_{g(t)}) \geq  1+ \delta,  \quad \forall \;
    t  \in [0, \infty).
 \end{align*}

    Then using the first part of the proof of
 Proposition 10.2 of~\cite{CT1}, we are able to prove that
 \begin{align*}
 \frac{1}{V}\int_M (\dot{\varphi} - c(t))^2 \omega_{\varphi}^2 < C_{\alpha}e^{-\alpha t}
 \end{align*}
 where $c(t)= \frac{1}{V} \int_M \dot{\varphi} \omega_{\varphi}^2$,
 $\alpha$ is some positive number.  As Riemannian curvature is
 uniformly bounded, we have a uniform Sobolev constant along this
 flow.   As in~\cite{CT1}, a parabolic Moser iteration
 then implies that $\dot{\varphi}$ is exponentially
 decaying, i.e.,
 \begin{align*}
     \dot{\varphi} \leq C_1 e^{-\alpha t}.
 \end{align*}
 The right hand side is an integrable function on $[0, \infty)$. It
 follows that $\varphi$ is uniformly bounded along the flow. By virtue of
 Yau's estimate, we can show that all higher
 derivatives of $\varphi$ are uniformly bounded in a fixed gauge.
 In particular, for every $t_i \to \infty$, we have subsequence
 $t_{i_k}$ such that
 \begin{align*}
     \omega + \st \varphi_{\alpha \bar{\beta}}(t_{i_k})   \to \omega
     + \st \varphi_{\alpha \bar{\beta}}(t_{\infty})
 \end{align*}
 for some smooth function $\varphi(t_{\infty})$ and $(M, \omega + \st \varphi_{\alpha \bar{\beta}}(t_{\infty}))$
 is a KE metric.  This argument is standard. Readers are referred
 to~\cite{CT1},~\cite{CT2},~\cite{PSSW1},~\cite{PSSW2},~\cite{CW}
 for more details.

  Let $\Psi_{i} \equiv id, h=\hat{g}$. We prove Theorem~\ref{theoremin: bounded}
 whenever $4 \geq c_1^2(M) \geq 1$.
 \end{enumerate}

 \end{proof}

\section{Properties of Deepest Bubbles}
 In this section, we discuss the behavior of \KRf on Fano surfaces at maximal
 curvature points whenever the Riemannian curvature is not uniformly
 bounded along the flow. As a corollary, we prove
 Theorem~\ref{theoremin: bubble}.

 Let $\{(M,g(t)), 0 \leq t < \infty \}$ be a K\"{a}hler Ricci flow
 solution on a Fano surface $M$. If the Riemannian curvature is not
 uniformly bounded, then there is a sequence of points $(x_i,t_i)$
 satisfying $Q_i=|Rm|_{g(t_i)}(x_i) \to \infty$ and
 \begin{align*}
 |Rm|_{g(t)}(x) \leq Q_i, \; \forall \; (x,t) \in M \times [0, t_i].
 \end{align*}
 After rescaling, we have
 \begin{align*}
 |Rm|_{g_i(t)}(x) \leq 1, \; \forall \; (x,t) \in M \times
 [-Q_it_i,0],
 \end{align*}
 where $g_i(t) = Q_i g(Q_i^{-1}t + t_i)$.
 Shi's estimate(~\cite{Shi1}, \cite{Shi2}) along Ricci flow implies that all the derivatives of
 Riemannian curvature are uniformly bounded.
 Therefore, we have the convergence
 \begin{align*}
 \{(M, x_i,  g_i(t)), -\infty < t \leq 0 \}
 \stackrel{C^{\infty}}{\to} \{(X, x_\infty, g_\infty(t)),
 -\infty < t \leq 0 \}.
 \end{align*}
 Here $C^{\infty}$ means that this convergence is in the Cheeger-Gromov sense.
 In particular, we have
 $ (M, x_i,  g_i(0)) \stackrel{C^{\infty}}{\to} (X, x_\infty,  g_\infty(0))$.

 Every $g_i(t)$ satisfies the following conditions
 \begin{align*}
     \D{g_i}{t} =  Q_i^{-1}g_i(t) - Ric_{g_i(t)}, \quad
     \sup_{M \times [-Q_it_i, 0]}  |R_{g_i(t)}(x)| \leq CQ_i^{-1}.
 \end{align*}
 So the limit flow  $\{(X, x_\infty, g_\infty(t)), -\infty < t \leq 0 \}$
 satisfies the equations
 \begin{align*}
     \D{g_{\infty}}{t} = - Ric_{g_{\infty}(t)}, \quad
     R_{g_{\infty}(t)} \equiv 0.
 \end{align*}
 It is an unnormalized Ricci flow solution, so the scalar curvature
 satisfies the equation  $\triangle R = \frac12 \triangle R +
 |Ric|^2$.   It forces that the limit solution is Ricci flat.

 As every manifold $(M, g_i(0))$ is a K\"ahler manifold with complex
 structure $J$ satisfying $\nabla_{g_i(0)} J=0$,  there is a limit
 complex structure $J_{\infty}$ such that $\nabla_{g_{\infty}(0)}
 J_{\infty}=0$.   Therefore, $(X, g_{\infty}(0), J_{\infty})$ is a
 Ricci flat K\"ahler manifold.

 For simplicity of notation, we
  denote $g_{\infty}$ as $g_{\infty}(0)$.

 \begin{definition}
   We call such a limit $(X, g_{\infty}, J_{\infty})$ as a deepest
   bubble.
 \end{definition}

 Since Riemannian curvature's $L^2$ norm is a rescaling invariant, we
 have
 \begin{align*}
    \int_{X}  |Rm|_{g_\infty}^2 d\mu_{g_\infty} & \leq \limsup_{i \to
    \infty} \int_{M} |Rm|_{g_i(0)}^2 d\mu_{g_i(0)} \\
    & =  \limsup_{i \to \infty} (\int_{M} |R|_{g_i(0)}^2
    d\mu_{g_i(0)} + C(M))\\
    & =  \limsup_{i \to \infty} (\int_{M} |R|_{g(t_i)}^2
    d\mu_{g(t_i)} + C(M))\\
    & \leq C_0.
 \end{align*}
 In the last step, we use Perelman's estimate that scalar curvature
 is uniformly bounded.    From  the noncollapsing property of this
 flow (Proposition~\ref{proposition: perelman}),
 the limit manifold $(X, g_{\infty})$ must be
 $\kappa$-noncollapsing on all scales, i.e.,
 \begin{align*}
    \frac{\Vol(B(p, r))}{r^4} \geq \kappa>0, \quad \forall \; p \in
    X, \; r>0.
 \end{align*}
 As Ricci curvature is bounded,  this inequality implies that
 $(X, g_{\infty})$ has uniform Sobolev constant.
 Therefore $(X, g_{\infty})$ is a Ricci flat manifold with
 bounded energy and uniform Sobolev constant.   Such a manifold must
 be an an Asymptotically Locally
 Euclidean (ALE) space.  The detailed proof can be found in either~\cite{An89}
  or~\cite{BKN},~\cite{Tian90}.
 So we have the following property.

 \begin{proposition}
   Every deepest bubble is a K\"ahler, Ricci flat Asymptotically Locally Euclidean (ALE)
   space.
 \end{proposition}

  Moreover, the fundamental group of every deepest bubble must be finite.

 \begin{proposition}
   If $Y^4$ is a non flat ALE space with flat Ricci curvature, then $\pi_1(Y)$ is finite.
 \end{proposition}

   This is proved in~\cite{Ant}.   For the simplicity of readers, we
   write down a simple proof here.

   \begin{proof}
      Fix $p \in Y$ and let $\pi: \tilde{Y} \to Y$ be the universal covering map.

     We argue by contradiction.  Suppose $\pi_1(Y)$ is a infinite group,
     then we can find a sequence of points $p_0, p_1, p_2, p_3, \cdots \in \pi^{-1}(p)$.
      Let $\gamma_i$ be the shortest geodesic connecting $p_0$ and
      $p_i$, $q_i$ be the center point of $\gamma_i$, $v_i$ be the
      tangent direction represented by $\gamma_i$ in $T_{q_i}
      \tilde{Y}$.

      Since $Y$ is an ALE space, there is a constant $R>0$ such that $Y$
      is diffeomorphic to $B(p, R)$.  Therefore, every  loop
      $\pi(\gamma_i)$ can be smoothly deformed into a loop in
      $B(p, R)$.  Consequently the shortest distance property of
      $\gamma_i$ assures that $\pi(\gamma_i)$ must locate in
      $\overline{B(p, 2^m R)}$.  As $\overline{B(p, 2^m R)}$ is a compact space,
      we may assume that
      $(\pi(q_i), \pi_*(v_i))$ converges to a point $(q, v) \in TX$.
      Remember $q \in \overline{B(p, 2^m R)}$ and $v$ is a unit vector in
      $T_q Y$.

      Let $(x, u) \in T\tilde{Y}$ such that $(\pi(x), \pi_*(u))=(q, v)$.
      Then there is deck transformation $\sigma_i \in \pi_1(Y)$
      such that $(\sigma_i(q_i), (\sigma_i)_* (v_i))  \to (x, u)$.
      $\sigma_i(\gamma_i)$ is clearly a shortest geodesic connecting
      $\sigma_i(p_0)$ and $\sigma_i(p_i)$ and  $\sigma_i(q_i)$ is the
      center of $\sigma_i(\gamma_i)$. Now there are only two
      possibilities.

\noindent
  \textit{case1. $\displaystyle \limsup_{i \to \infty} |\gamma_i|=|\sigma_i(\gamma_i)|=\infty$ }

    In this case, by taking subsequence if necessary, we can assume that
    $\sigma_i(\gamma_i)$ tends to a line passing
    through $q_0$.  As $\tilde{Y}$ has flat Ricci curvature, it
    splits a line. So $\tilde{Y}= N^3 \times \R$.  $N^3$ must be  Ricci flat
    and therefore Riemannian flat. This implies $\tilde{Y}$
    and $Y$ are flat.   According to the assumption of $Y$, this is
    impossible.

\noindent
   \textit{case2. $\displaystyle |\gamma_i|<C$ uniformly. }

    Since all $p_i$'s
    locate in $\overline{B(p_0, C)} \subset \tilde{Y}$ which is a compact set.
    Therefore for every small $\epsilon>0$, there exists $i, j$ such
    that $d(p_i, p_j)<\epsilon$.  It follows that $Y$ contains a
    geodesic lasso passing through $p$ and its length is less than
    $\epsilon$ no matter how small $\epsilon$ is.   Of course this
    will not happen on a smooth manifold $Y$.\\

  Since both cases will not happen, our assumption must be wrong.
  Therefore $\pi_1(Y)$ is finite.
  \end{proof}

  Let $\tilde{X} \stackrel{\pi}{\rightarrow} X$ be the universal covering.
 Then $(\tilde{X}, \tilde{g}, \tilde{J})$
 is a K\"ahler Ricci flat manifold, where $\tilde{g}=\pi^*(g_{\infty})$, $\tilde{G}=\pi_* \circ J \circ (\pi_*)^{-1}$. So its holonomy group is
 $SU(2)=Sp(1)$.  Therefore, $\tilde{X}$ has a hyper-K\"ahler structure
 by Berger's classification.   Since $\pi_1(X)$ is finite, we have
 \begin{align*}
 \int_{\tilde{X}} |Rm|_{\tilde{g}}^2 d\mu_{\tilde{g}}
 =|\pi_1(X)| \int_X |Rm|_{g_{\infty}}^2 d\mu_{g_{\infty}} < \infty.
 \end{align*}
 It follows that $\tilde{X}$ is an ALE space as
 we argued before.  This means that $\tilde{X}$ is a hyper-K\"ahler
 ALE space.  However, all these hyper K\"ahler ALE spaces has been classified by
 Kroheimer in~\cite{Kr89}.

 \begin{proposition}[Kroheimer]
      Let $\Gamma$ be a finite subgroup of $SU(2)$ and $\pi: M \to \C^2 /
      \Gamma$ be the minimal resolution of the quotient space $\C^2/
      \Gamma$ as a complex variety. Suppose that three cohomology
      classes $\alpha_I, \alpha_J, \alpha_K \in H^2(M;\R)$ satisfy
      the non-degeneracy condition for each $\Sigma \in H_2(M; \Z)$
      with $\sigma \cdot \sigma = -2$, there exists $A \in \{ I, J,
      K\}$ with $\alpha_A(\Sigma) \neq 0$.

 Then there exists an ALE Riemannian metric $g$ on $M$ of order $4$
 together with a hyper K\"{a}hler structure $(I,J,K)$ for which the
 cohomology class of the K\"{a}hler form $[\omega_A]$ determined by
 the complex structure $A$ is given by $\alpha_A$ for all $A \in I,J,K$.
 Conversely every hyper K\"{a}hler ALE 4-manifold of order 4
 can be obtained as above.
  \label{proposition: kroheimer}
 \end{proposition}

 As a corollary of this property, we know $\tilde{X}$ must be diffeomorphic to a
 minimal resolution of $\C^2 / \Gamma$ for some finite group $\Gamma \subset
 SU(2)$.\\

  Now let's consider the property of $X$ by its submanifolds.
 On the deepest bubble $(X, g_{\infty}, J_{\infty})$, let
 $\omega_{\infty}$ be the metric form determined by $g_{\infty}$ and
 $J_{\infty}$.

 \begin{proposition}
 For every closed 2-dimensional submanifold $C$ of $X$, we have
 \begin{align*}
    \int_{C} \omega_{\infty} =0.
 \end{align*}
 In particular, $(X, g_{\infty}, J_{\infty})$ doesn't contain any
 compact holomorphic curve.
 \label{proposition: noholomorphic}
 \end{proposition}

 \begin{proof}
  Fix a  closed 2-dimensional submanifold $C \subset X$. By the smooth
  convergence property, we know there is a sequence of closed 2-dimensional
  smooth manifolds $C_i \subset M$ such that
 \begin{align}
    (C_i, g_i(0)|_{C_i}) &\stackrel{C^\infty}{\to} (C, g_{\infty}|_{C}
    ); \notag \\
    (C_i, \omega_i(0)|_{C_i}) &\stackrel{C^\infty}{\to} (C, \omega_{\infty}|_{C}). \label{eqn: convergence}
 \end{align}
 Therefore,
 \begin{align*}
  Area_{g_{\infty}}(C) = \lim_{i \to \infty} Area_{g_i(0)}(C_i).
 \end{align*}
 Use Wirtinger's inequality, we get
 \begin{align*}
   |\int_{C_i} \omega_i(0) |&=| \int_{C_i} \cos \alpha d\mu_{g_i(0)}| \\
   &\leq
   \int_ {C_i} |\cos \alpha| d\mu_{g_i(0)}\\
   &\leq \int_{C_i} d\mu_{g_i(0)} = Area_{g_i(0)}(C_i)
 \end{align*}
 where $\alpha$ is the K\"{a}hler angle.  Consequently, for large
 $i$, the following inequality hold
 \begin{align*}
   |\int_{C_i} \omega_i(0)| \leq 2 Area_{g_{\infty}} (C).
 \end{align*}
 On the other hand, we know $[\omega_i(0)]= Q_i c_1(M)$, this tells
 us that
 \begin{align}
   \int_{C_i} \omega_i(0) = Q_i \int_{C_i} c_1(M) =Q_i a_i
 \label{eqn: intersection}
 \end{align}
 where we denote $a_i = \int_{C_i} c_1(M) \in \Z$.  So we have
 inequality
 \begin{align*}
    Q_i|a_i| \leq 2Area_{g_{\infty}} C
 \end{align*}
 or
 \begin{align*}
  |a_i| \leq \frac{2 Area_{g_{\infty}} C}{Q_i}.
 \end{align*}

 Note that $Area_{g_{\infty}} C$ is a fixed number and $Q_i \to \infty$, so
 $|a_i| \to 0$. However, $a_i$ are integers. This forces that for
 large $i$, $a_i \equiv 0$.  Now we go
  back to equality (\ref{eqn: intersection})
 and see for large $i$, we have
 \begin{align*}
    \int_{C_i} \omega_i \equiv 0.
 \end{align*}
 Using smooth convergence property,
 equation (\ref{eqn: convergence}) yields
 \begin{align}
   \int_{C} \omega_{\infty} = \lim_{i \to \infty} \int_{C_i}
   \omega_i(0) =0.
 \label{eqn: cruciallimit}
 \end{align}
 \end{proof}

 \begin{remark}
   The reason for no compact divisors in $X$ is that Ricci flow evolves metric forms
 continuously, so
 it cannot change the class which is discrete.    This should be some common
 phenomenon in geometric flow.  For example,  in~\cite{Sj}, Jeff Streets proved that on a nontrivial bundle, the
 base manifold of a blowup limit along a renormalization group flow
 must be noncompact.
 \end{remark}

 Combining the previous propositions, we can conclude this section
 by Theorem~\ref{theoremin: bubble}.

 \section{Bound Riemannian Curvature of Flow on Toric Fano Surfaces}
 This section is devoted to the proof of
 Theorem~\ref{theoremin: toric}.
 \begin{lemma}
    If $(X, g_{\infty}, J_{\infty})$ is a  bubble coming out of a \KRF solution
    with toric symmetric metrics, then $(X, g_{\infty}, J_{\infty})$ is also
    a toric surface.  Moreover, $X$ is simply connected,
    $b_2(X)>0$ and $H_2(X)$
    is generated by holomorphically embedded $\CP_1$'s in $X$.
 \label{lemma: toricdivisor}
 \end{lemma}

 \begin{proof}
    According to the construction of $X$, we have the following
    convergence in Cheeger-Gromov sense
 \begin{align*}
      (M, x_i, g_i(0)) \sconv (X, x_{\infty}, g_{\infty}).
 \end{align*}
 Since the toric symmetry property will be preserved under \KRF, we see that
 every metric $g_i(0)=Q_ig(t_i)$ is a toric symmetric metric.     Using exactly
 the statement of Proposition 16 of~\cite{CLW}, we know that $X$
 is a toric surface with nontrivial $H_2(X)$. Moreover, $H_2(X)$ are
 generated by holomorphic $\CP^1$'s in $X$.    Actually, according
 to this proof, there is a Morse function defined on $X$.
 Furthermore, every critical point of $X$ has even indices.  Therefore $X$ is homotopic to
 a $CW$-complex with only cells of even dimension.  Consequently,
 $X$ must be simply connected by considering the Euler
 characteristic number of $X$.
 \end{proof}

 Now we are able to prove Theorem~\ref{theoremin: toric}.
 \begin{proof}
   We only need to show Riemannian curvature is uniformly bounded.
   If Riemannian curvature is not bounded, then we can blowup a toric
   deepest bubble $X$.  According to Proposition~\ref{proposition: noholomorphic}, it
   doesn't contain any compact divisor. On the other hand,
   Lemma~\ref{lemma: toricdivisor} implies that $X$ must contain
   a $\CP^1$ as a compact divisor. Contradiction!
 \end{proof}

 \begin{remark}
 If $M \sim \CP^2$ or $\Blow{}$,
 Theorem~\ref{theoremin: toric} can be proved in a different
 way.    Suppose Riemannian curvature is unbounded. Then  we can obtain a bubble
 $(X, g_{\infty}, J_{\infty})$ which is simply connected. Therefore itself is a hyper
 K\"ahler ALE space.  By Kroheimer's classification (Proposition~\ref{proposition: kroheimer}),
 $X$ must contain a
 compact smooth 2-dimensional submanifold $C$ whose self intersection number is
 $-2$. By the smooth  convergence property, we can get a sequence of
 closed smooth 2-dimensional smooth manifolds $C_i \subset M$ such
 that $(C_i, g_i(0)|_{C_i}) \sconv (C, h|_C)$. In particular $C_i$ is
 diffeomorphic to $C$ and a small tubular neighborhood of $C_i$ is
 diffeomorphic to a small tubular neighborhood of $C$. Therefore,
 $[C_i] \cdot [C_i] = [C] \cdot [C] =-2$.  Note that $[C_i] \in
 H_2(M, \Z)$.   However, $H_2(\CP^2, \Z)$ and $H_2(\Blow{}, \Z)$
 doesn't contain any element of self intersection number $-2$. So we obtain a
 contradiction!
 \end{remark}

 \begin{remark}
  Theorem~\ref{theoremin: toric} needs the condition that initial
  metric $g(0)$ is toric symmetric. Natasa Sesum conjectured that
  the toric symmetry condition is not necessary.  In our proof, the
  symmetry is only
  a technical condition, we believe that it can be dropped. Actually, we
  believe that starting from any metric in canonical class, \KRf will
  evolve the metrics into a KRS on every toric manifold. It will
  be discussed in a subsequent paper.
 \end{remark}

 \vspace{0.5in}

 Xiuxiong Chen,  Department of Mathematics, University of
 Wisconsin-Madison, Madison, WI 53706, USA; xiu@math.wisc.edu\\

 Bing  Wang, Department of Mathematics, University of Wisconsin-Madison,
 Madison, WI, 53706, USA; bwang@math.wisc.edu

 \qquad \qquad \quad \; Department of Mathematics, Princeton University,
  Princeton, NJ 08544, USA; bingw@math.princeton.edu

\end{document}